\documentclass[11pt,leqno]{article}
\usepackage{graphicx, amsfonts, amsthm, amsxtra, amssymb, verbatim, makeidx}
\usepackage{subeqnarray, relsize}
\usepackage[mathscr]{euscript}
\usepackage{hyperref}
\makeindex
\makeglossary
\begin{document}
\newtheorem{theo}{Theorem}
\newtheorem{exam}{Example}
\newtheorem{coro}{Corollary}
\newtheorem{defi}{Definition}
\newtheorem{prob}{Problem}
\newtheorem{lemm}{Lemma}
\newtheorem{prop}{Proposition}
\newtheorem{rem}{Remark}
\newtheorem{conj}{Conjecture}
\newtheorem{calc}{}

\def\gru{\mu} 
\def\pg{{ \sf S}}               
\def\TS{{\mathlarger{\bf T}}}                
\def\NB{{\mathlarger{\bf N}}}
\def\group{{\sf G}}
\def\NLL{{\rm NL}}   

\def\plc{{ Z_\infty}}    
\def\pola{{u}}      
\newcommand\licy[1]{{\mathbb P}^{#1}} 
\newcommand\aoc[1]{Z^{#1}}     
\def\HL{{\rm Ho}}     
\def\NLL{{\rm NL}}   

\def\Z{\mathbb{Z}}                   
\def\Q{\mathbb{Q}}                   
\def\C{\mathbb{C}}                   
\def\N{\mathbb{N}}                   
\def\uhp{{\mathbb H}}                
\def\A{\mathbb{A}}                   
\def\dR{{\rm dR}}                    
\def\F{{\cal F}}                     
\def\Sp{{\rm Sp}}                    
\def\Gm{\mathbb{G}_m}                 
\def\Ga{\mathbb{G}_a}                 
\def\Tr{{\rm Tr}}                      
\def\tr{{{\mathsf t}{\mathsf r}}}                 
\def\spec{{\rm Spec}}            
\def\ker{{\rm ker}}              
\def\GL{{\rm GL}}                
\def\ker{{\rm ker}}              
\def\coker{{\rm coker}}          
\def\im{{\rm Im}}               
\def\coim{{\rm Coim}}            
\def\p{{\sf  p}}
\def\U{{\cal U}}   

\def\weig{{\nu}}
\def\r{{ r}}                       
\def\k{{\sf k}}                     
\def\ring{{\sf R}}                   
\def\X{{\sf X}}                      
\def\Ua{{   L}}                      
\def\T{{\sf T}}                      
\def\asone{{\sf A}}                  

\def\Ts{{\sf S}}
\def\cmv{{\sf M}}                    
\def\BG{{\sf G}}                       
\def\podu{{\sf pd}}                   
\def\ped{{\sf U}}                    
\def\per{{\bf  P}}                   
\def\gm{{  A}}                    
\def\gma{{\sf  B}}                   
\def\ben{{\sf b}}                    

\def\Rav{{\mathfrak M }}                     
\def\Ram{{\mathfrak C}}                     
\def\Rap{{\mathfrak G}}                     

\def\mov{{\sf  m}}                    
\def\Yuk{{\sf C}}                     
\def\Ra{{\sf R}}                      
\def\hn{{ h}}                         
\def\cpe{{\sf C}}                     
\def\g{{\sf g}}                       
\def\t{{\sf t}}                       
\def\pedo{{\sf  \Pi}}                  

\def\Der{{\rm Der}}                   
\def\MMF{{\sf MF}}                    
\def\codim{{\rm codim}}                
\def\dim{{\rm    dim}}                
\def\Lie{{\rm Lie}}                   
\def\gg{{\mathfrak g}}                

\def\u{{\sf u}}                       

\def\imh{{  \Psi}}                 
\def\imc{{  \Phi }}                  
\def\stab{{\rm Stab }}               
\def\Vec{{\rm Vec}}                 
\def\prim{{\rm  0}}                  

\def\Fg{{\sf F}}     
\def\hol{{\rm hol}}  
\def\non{{\rm non}}  
\def\alg{{\rm alg}}  
\def\tra{{\rm tra}}  

\def\bcov{{\rm \O_\T}}       

\def\leaves{{\cal L}}        

\def\cat{{\cal A}}              
\def\im{{\rm Im}}               

\def\pn{{\sf p}}              
\def\Pic{{\rm Pic}}           
\def\free{{\rm free}}         
\def \NS{{\rm NS}}    
\def\tor{{\rm tor}}
\def\codmod{{\xi}}    

\def\GM{{\rm GM}}

\def\perr{{\sf q}}        
\def\perdo{{\cal K}}   
\def\sfl{{\mathrm F}} 
\def\sp{{\mathbb S}}  

\newcommand\diff[1]{\frac{d #1}{dz}} 
\def\End{{\rm End}}              

\def\sing{{\rm Sing}}            
\def\cha{{\rm char}}             
\def\Gal{{\rm Gal}}              
\def\jacob{{\rm jacob}}          
\def\tjurina{{\rm tjurina}}      
\newcommand\Pn[1]{\mathbb{P}^{#1}}   
\def\P{\mathbb{P}}
\def\Ff{\mathbb{F}}                  

\def\O{{\cal O}}                     

\def\ring{{\mathsf R}}                         
\def\R{\mathbb{R}}                   

\newcommand\ep[1]{e^{\frac{2\pi i}{#1}}}
\newcommand\HH[2]{H^{#2}(#1)}        
\def\Mat{{\rm Mat}}              
\newcommand{\mat}[4]{
     \begin{pmatrix}
            #1 & #2 \\
            #3 & #4
       \end{pmatrix}
    }                                
\newcommand{\matt}[2]{
     \begin{pmatrix}                 
            #1   \\
            #2
       \end{pmatrix}
    }
\def\cl{{\rm cl}}                

\def\hc{{\mathsf H}}                 
\def\Hb{{\cal H}}                    
\def\pese{{\sf P}}                  

\def\PP{\tilde{\cal P}}              
\def\K{{\mathbb K}}                  

\def\M{{\cal M}}
\def\RR{{\cal R}}
\newcommand\Hi[1]{\mathbb{P}^{#1}_\infty}
\def\pt{\mathbb{C}[t]}               
\def\gr{{\rm Gr}}                
\def\Im{{\rm Im}}                
\def\Re{{\rm Re}}                
\def\depth{{\rm depth}}
\newcommand\SL[2]{{\rm SL}(#1, #2)}    
\newcommand\PSL[2]{{\rm PSL}(#1, #2)}  
\def\Resi{{\rm Resi}}              

\def\L{{\cal L}}                     
\def\Aut{{\rm Aut}}              
\def\any{R}                          
\newcommand\ovl[1]{\overline{#1}}    

\newcommand\mf[2]{{M}^{#1}_{#2}}     
\newcommand\mfn[2]{{\tilde M}^{#1}_{#2}}     

\newcommand\bn[2]{\binom{#1}{#2}}    
\def\ja{{\rm j}}                 
\def\Sc{\mathsf{S}}                  
\newcommand\es[1]{g_{#1}}            
\newcommand\V{{\mathsf V}}           
\newcommand\WW{{\mathsf W}}          
\newcommand\Ss{{\cal O}}             
\def\rank{{\rm rank}}                
\def\Dif{{\cal D}}                   
\def\gcd{{\rm gcd}}                  
\def\zedi{{\rm ZD}}                  
\def\BM{{\mathsf H}}                 
\def\plf{{\sf pl}}                             
\def\sgn{{\rm sgn}}                      
\def\diag{{\rm diag}}                   
\def\hodge{{\rm Hodge}}
\def\HF{{ F}}                                
\def\WF{{ W}}                               
\def\HV{{\sf HV}}                                
\def\pol{{\rm pole}}                               
\def\bafi{{\sf r}}
\def\id{{\rm id}}                               
\def\gms{{\sf M}}                           
\def\Iso{{\rm Iso}}                           

\def\hl{{\rm L}}    
\def\imF{{\rm F}}
\def\imG{{\rm G}}

\def\cf{r}   
\def\cm{\checkmark}
\def\MI{{\cal M}}
\def\se{{\sf s}}
\def\codnum{{\sf C}}

\begin{center}
{\LARGE\bf Computing the lines of a smooth cubic surface
}
\\
\vspace{.25in} {\large {\sc Hossein Movasati}}
\footnote{
Instituto de Matem\'atica Pura e Aplicada, IMPA, Estrada Dona Castorina, 110, 22460-320, Rio de Janeiro, RJ, Brazil,
{\tt \href{http://w3.impa.br/~hossein/}{www.impa.br/$\sim$hossein}, hossein@impa.br.}}
\end{center}
\begin{abstract}
We give an explicit formula for the $27$ lines of a smooth cubic surface near the Fermat surface. Our formula involves convergent power series  with coefficients in the extension of rational numbers with the sixth root of unity. Our main tool is the Artinian Gorenstein ring of socle two attached to such lines.     
\end{abstract}

\section{Introduction}
One of  the most well-known and beautiful objects in classical algebraic geometry is the  $27$ lines of a smooth cubic surface $X$. This has been partially  justified by plasters of cubic surfaces with their  lines made by mathematicians in the 19th century. For historical account on this and the visualization of cubic surface with their lines  see \cite{StratenLabs}. Algorithms to compute such lines are mainly based on brute force substitution of equations of lines in the equation of $X$. As the Hodge decomposition of the second cohomology of $X$ consists only of the middle piece, the study of lines of $X$ using Hodge theory might seem hopeless.
However, it turns out that the  Artinian Gorenstein ring attached to Hodge cycles originated from the works of P. Griffiths in 1970's and further elaborated in  \cite{voisin89, Otwinowska2003, Dan-2014, Roberto,  ho13-Roberto} can be useful in order to write down the equation of simple algebraic cycles like lines, using two dimensional periods of $X$. This idea has been explained in \cite{EmreHossein2018} which uses approximation of periods with a high precision in \cite{Sertoz2018}.   
Near to the Fermat variety the author in \cite{ho13} has  written down explicit formulas for the Taylor  series of such periods which leads us to the main result  Theorem \ref{main} of the present paper.     

We write down a cubic surface in the format 
\begin{equation}
\label{clarakaren2019}
 X_t:\ \ F_t:=x_0^3+x_1^3+x_2^3+x_3^3-\sum_{i\in I}t_ix^i=0,\ \ t:=(t_i,\ i\in I)\in\T:=\C^{20}\backslash\{\Delta=0\}, 
\end{equation}
where $I$ is the set of exponents of monomials of degree $3$ in four variables $x_0,x_1,x_2,x_3$ and $\Delta=0$ is the loci of singular cubic surfaces. 
We have written this as perturbation of the Fermat surface $X_0$, as our main computations is done in  a neighborhood of Fermat. 
For $\beta\in \N_0^{4}$ we denote by $\beta_i$ its $(i+1)$-th coordinate, that is, $\beta=(\beta_0,\beta_1,\beta_2,\beta_3)$, and for $n\in\Z$, $\bar n\in\N_0$ 
is defined by the rules $0\leq \bar n \leq 2,\ \  n\equiv_{3}\bar n$. 
For a positive rational number $r$,  $[r]$ is the integer part of $r$,  that is $[r]\leq r<[r]+1$, $\{r\}:=r-[r]$ and   
$\langle r\rangle=(r-1)(r-2)\cdots (r-[r])$ (and hence $\langle r\rangle=0$ if $r\in\N$).
We consider the set of $({\bf m},{\bf n},{\bf l},\zeta_1,\zeta_2)$, where $({\bf m},  {\bf n},{\bf l})=(1,2,3),(2,1,3),(3,1,2)$ and 
$\zeta_1,\zeta_2$ are roots of $-1$, that is,
 $\zeta_1^3=\zeta_2^3=-1$.
This set consists of $27$ elements. 
In this article we prove the following theorem: 
\begin{theo}
 \label{main}
 For the twenty seven choice of  $k=({\bf m},{\bf n},{\bf l}, \zeta_1,\zeta_2)$ as above we have the following rational curve inside $X_t$: 
 \begin{equation}
\label{CarolinePilar2022}
 \P^1_{k,t}:  
\left\{
 \begin{array}{l}
c_{0212}\cdot x_0-c_{0202}\cdot x_1+c_{0201}\cdot x_2+0\cdot x_3=0\\
c_{0223}\cdot x_0+0\cdot x_1
-c_{0203}\cdot x_2+c_{0202}\cdot x_3=0
 \end{array}
 \right.,
 \end{equation} 
 where
 $$
c_{i_1i_2j_1j_2}=\det\begin{bmatrix}
                         p_{i_1j_1}, & p_{i_1j_2}\\
                         p_{i_2j_1}, & p_{i_2j_2}
                        \end{bmatrix},
$$ 
\begin{eqnarray}
 \label{15.12.16-2022}
 p_{ij}&= & 
\mathlarger{\mathlarger{\mathlarger{\sum}}}_{a: I\to \N_0}
\frac{1}{ a! } 
\zeta_1^{  
\overline{(\beta_{ij}+a^*)_0+1}   }\cdot \zeta_2^{\overline{(\beta_{ij}+a^*)_{\bf n}+1}}
\mathlarger{\prod}_{i=0}^{3} \left\langle\frac{(\beta_{ij}+a^*)_i+1}{3}\right\rangle \cdot  t^a,
\end{eqnarray}
the sum runs through all $\#I$-tuples $a=(a_\alpha,\ \ \alpha\in I)$
of non-negative integers such that  
\begin{equation}
\label{TheLastMistake2017}
\left\{\frac{ (\beta_{ij}+a^*)_{0}+1}{3} \right\}+   \left\{\frac{ (\beta_{ij}+a^*)_{\bf m}+1}{3} \right\}=1,\ \ \ 
\end{equation}
$$
\left\{\frac{ (\beta_{ij}+a^*)_{\bf n}+1}{3} \right\}+   \left\{\frac{ (\beta_{ij}+a^*)_{\bf l}+1}{3} \right\}=1,\ \ \ 
$$
and 
\begin{equation}
t^a:=\prod_{\alpha\in I}t_\alpha^{a_\alpha}, \ \ \ \ \  
a!:=\prod_{\alpha\in I}a_\alpha!,  \ \   \ \  a^* := \sum_{\alpha}a_\alpha\cdot \alpha.
\end{equation}
 and $\beta_{ij}\in\N_0^4$ is the exponent vector of $x_{i}x_j$, that is $x^{\beta_{ij}}=x_ix_j$.
\end{theo}
We can interpret this theorem as finding the rational curves of  $X_t$ as variety over the field of formal power series in $t$ with coefficients in $\Q(\zeta_6)$.  The power series $p_{ij}$'s  are Taylor series of periods of $X_t$, and so near $0$ they are convergent. As they are algebraic functions, one might compute their values over rational numbers by high precision in order to describe the corresponding algebraic numbers.

In 27 July 2022 I visited Duco van Straten for few hours in Mainz. Despite his eye surgery, his excitement with my Hodge theory book \cite{ho13} and my joint paper \cite{EmreHossein2018} had so powerful effect in my soul 
that I decided to elaborate further some of the ideas developed in these texts which resulted in the present text. My sincere thanks go to him for many encouraging discussions.

\section{Proof of Theorem \ref{main}}
\label{aima2017}
Let $X\subset \P^3$ be a smooth cubic surface given by the homogeneous polynomial $F$ of degree $3$.  
For a homogeneous polynomial $P$ of degree $2$ in $x_0,x_1,x_2,x_3$  define 
$$
\omega_{P}:={\rm Resi}\left(\frac{
P\cdot 
\sum_{i=0}^3 (-1)^ix_i\widehat{dx_i}
 }{f^{2}_t}\right)\in H^2_\dR(X),
$$
where ${\rm Resi}: H^3(\P^3\backslash X)\to H^2_\dR(X)$ is the Griffiths residue map, see for instance \cite[Chapter 7]{ho13-Roberto}.
The following theorem tells us how to recover the equations of $\P^1$ using its periods. 
\begin{prop}
Let $\P^1$ be a line inside $X$. 
The $4\times 4$ matrix
\begin{equation}
\label{18102022}
A=\begin{bmatrix}
   \mathlarger{\int}_{\P^1}\omega_{x_ix_j}
  \end{bmatrix}_{0\leq i,j\leq 3 }
\end{equation}
is of rank two and its kernel is generated by $a_i:=(a_{1,i},a_{2,i},a_{3,i},a_{4,i}),\ i=1,2$, where $a_{1,i}x_0+a_{2,i}x_1+a_{3,i}x_2+a_{4,i}x_3=0,\ i=1,2$ are two linear equations of $\P^1$. 
\end{prop}
\begin{proof}
The homology $H_2(X,\Z)$ is of rank $7$ and all cycles $\delta\in H_2(X,\Z)$ are Hodge cycles. Let $[Z_\infty]$ be the homology class of the hyperplane section. 
For every Hodge cycle $\delta\in H_2(X,\Z)/\Z [\plc]$ we define its associated Artinian Gorenstein ideal $I(\delta)\subset\C[x]$
and the the corresponding Artinian  Gorenstein algebra $R(\delta):=\C[x]/ I(\delta)$ which is of socle $2$. We have $I(\delta)_0=0, I(\delta)_a=\C[x]_a,\ a\geq 3$ and for $a=1,2$: 
$$
I(\delta)_a:=\left \{Q\in \C[x]_a \Bigg| \mathlarger{\int}_{\delta} \omega_{PQ} =0,\ \ 
\forall P\in \C[x]_{2-a} \right\}. 
$$
If $\delta=[\P^1],\ \ \P^1:=\{ f_1=f_2=0\}\subset X$ and $f_1,f_2$ are homogeneous degree one polynomials then we have $F=f_1g_1+f_2g_2$ for some degree two homogeneous polynomials $g_1,g_2$ and it turns out that 
$$
\langle f_1,f_2, g_1,g_2\rangle=I(\delta) 
$$
as we have the inclusion $\subset$ and both ideals are of the same socle $2$. This idea comes originally from \cite{Dan-2014} and has been further elaborated in  \cite[Chapter 11]{ho13-Roberto}. 
Therefore, $I(\delta)_1=\C f_1+\C f_2$, that is, we can recover the ideal of $\P^1$ from its Artinian Gorenstein ideal $I(\delta)_1$. It follows that the matrix $A$
 is of rank two  and the two linearly independent  equations of $\P^1\subset \P^3$ are give by 
 $a_ix=0$, where $a_i, i=1,2$ are two linearly independent vectors in the kernel of $A$. 
 \end{proof}
 
\begin{proof}[Proof of Theorem \ref{main}]
 We consider the family \eqref{CarolinePilar2022} and the the matrix $A=A(t)$ in \eqref{18102022} has entries which are holomorphic functions in $t\in(\T,0)$. The matrix $A(t)$ evaluated at the Fermat point $t=0$ is
 $$
 A(0)=\frac{2\pi\sqrt{-1}\zeta_1\zeta_2}{9}
 \begin{bmatrix}
  0& 0& \zeta_1\zeta_2 & \zeta_1 \\
  0& 0&        \zeta_2 &  1 \\
\zeta_1\zeta_2 & \zeta_2 & 0 & 0 \\
\zeta_1 & 1&     0 &  0 \\
 \end{bmatrix}.
 $$
 This has been calculated on   \cite[Theorem 1]{ho13-Roberto} and it implies that the first and third row of $A(t)$ are linear independent for $t\in(\T,0)$. We have to find two vectors perpendicular to these two vectors. For this we use  the fact that the determinant of the $3\times 3$ minors of $A(t)$ formed by rows $0,2,1$ and columns $0,1,2$ (resp.  rows $0,2,3$ and columns $0,2,3$) are zero. We get two vectors corresponding to coefficients of \eqref{CarolinePilar2022}. Note that $c_{0202}(0)\not=0$.  

 For the Fermat variety $X_0$, the twenty seven lines are given by 
\begin{equation}
\label{CarolinePilar}
 \P^1_{k,0}:  
\left\{
 \begin{array}{l}
 x_{0}-\zeta_1x_{\bf m}=0,\\
 x_{\bf n}-\zeta_2 x_{\bf l}=0,
 \end{array}
 \right.\ \ \ \ \ \zeta_1^3=\zeta_2^3=-1, \{0,{\bf m},{\bf n},{\bf l}\}=\{0,1,2,3\},\ {\bf n}<{\bf l}. 
 \end{equation}
For $t\in\T$ near to the Fermat point $0$, there is a unique rational curve $\P^1_{k, t}$ which is obtained by deformation of $\P^1_{k,0}$. We want to compute its equations.  Its homology class $\delta_t=[\P^1_{k,t}]\in H_2(X_t,\Z)$ is the monodromy (parallel transport) of the homology class of $\P^1_{k,0}$.  The Taylor series of integration of $\omega_{x_ix_j}$ over $\delta_t$ at $t=0$ is computed in \cite[\S 18.3]{ho13} and we have 
 \begin{eqnarray}
 \label{15.12.16}
 & & 
 \frac{ -6}{ 2\pi \sqrt{-1}}
 \mathlarger{\mathlarger{\int}}_{\delta_t}\omega_{x_ix_j}
 =p_{ij},\ \ i,j=0,1,2,3,
\nonumber
\end{eqnarray}
where $p_{ij}$ is the power series \eqref{15.12.16-2022}. This together with Proposition \ref{18102022} finishes the proof. 
 \end{proof}
 \begin{rem}\rm 
For the family \eqref{CarolinePilar2022} define 
$$
P_{ij}(y):=\prod_{k=1}^{27}(y- \int_{\P^1_{k,t}}\omega_{ij})=y^{27}+\sum_{i=1}^{27}\tilde p_{i}(t)y^{27-i},
$$
where $\P^1_{k,t}, \ k=1,2,\ldots,27$ are lines of $X_t$. The coefficients $\tilde p_i$ of this polynomial in $y$ are rational functions in $t$ with coefficients in $\Q$ and poles along the discriminant. This follows from the finite growth of integrals near degeneration points $\Delta=0$, see \cite[Chapter 13]{arn}, and simple Galois action argument.   The computations of these coefficients seems to be out of reach even for the following family written in the affine chart $x_0=1$: 
\begin{equation}
\label{clarakaren2019-lubo}
 X_t:\ \ f_t:=x_1^3+x_2^3+x_3^3-\sum_{i\in I}t_ix^i=0,\ \ t:=(t_i,\ i\in I)\in\T:=\C^{10}\backslash\{\Delta=0\}, 
\end{equation}
where $x^i$'s are monomials of degree $\leq 2$ in $x_1,x_2,x_3$. This is a tame polynomial in the sense of \cite[Chapter 10]{ho13}. 
We consider the $\C^*$-action 
$$
\C^3\times\C^*\to\C^3, \ ((x_1,x_2,x_3),a) \to (a^{-1}x_1,a^{-1}x_2,a^{-1}x_3)
$$ 
which induces an action in $\T$:
$$
(t_\alpha,\alpha\in I)\bullet a=(t_\alpha a^{3-\deg(x^\alpha)},\alpha\in I)
$$
and an isomorphism $h: X_{t\bullet a}\to X_{t}$ such that for a monomial $x^\beta$ with $\deg(x^\beta) \leq 2$,  we have 
$h^*\frac{x^\beta dx}{f_t^2}=a^{3-\deg(x^\beta)}\frac{x^\beta dx}{f_t^2}$. We conclude that
\begin{equation}
\int_{\delta_{t\bullet a}} \frac{x^\beta dx}{f_t^2}=a^{\deg(x^\beta)-3}\int_{\delta_{t}} \frac{x^\beta dx}{f_t^2} 
\end{equation}
which implies that this integral is homogeneous of negative degree $\deg(x^\beta)-3$ in variables $t_\alpha, \deg(t_\alpha)=3-\deg(x^\alpha)$. The discriminant $\Delta$ is a homogeneous polynomial of degree $(d-1)^{n+1}d=2^33=24$, see \cite[Section 10.9]{ho13}. The pole order of $\tilde p_i$ along $\Delta$ is at most $2i$, see \cite[Equality 10.25 and  Proposition 11.2]{ho13}.
If this is the case then $\tilde p_i=\frac{p_i}{\Delta^{2i}}$, $\tilde p_i$ a homogeneous polynomial, and 
$\deg(p_i)-2i\deg(\Delta)=i(\deg(x^\beta)-3)$. Therefore,  $p_i$ is of  degree $(45+\deg(x^\beta))i$.
In theory, one can use the formal power series $p_{ij}$ in \eqref{15.12.16-2022} and compute the polynomials $p_i$, however, in practice this is out of the capacity of the author's simple computer code written in Singular. This can be found in the tex file of the present text in arxiv.  
\end{rem}

\def\cprime{$'$} \def\cprime{$'$} \def\cprime{$'$} \def\cprime{$'$}

\end{document}